\documentclass{amsart}
\usepackage{cite}

\usepackage[english]{babel}
\usepackage{amssymb}
\usepackage{latexsym}
\usepackage{xypic}
\usepackage{multirow}
\usepackage{float}
\usepackage{graphicx}
\usepackage[all]{xy}
\usepackage[caption = false]{subfig}
\usepackage{url}

\newtheorem{theorem}{Theorem}[section]
\newtheorem{lemma}[theorem]{Lemma}
\newtheorem{proposition}[theorem]{Proposition}

\newtheorem{corollary}[theorem]{Corollary}

\theoremstyle{definition}
\newtheorem{definition}[theorem]{Definition}

\theoremstyle{remark}

\numberwithin{equation}{section}

\begin{document}

\title{\sc{Ribbon graphs and the fundamental group of surfaces}}

\author{Rodrigo D\'avila Figueroa}
\address{ IMATE UNAM, Unidad Cuernavaca, Av. Universidad s/n. Col. Lomas de Chamilpa,
C.P. 62210, Cuernavaca, Morelos, M\'exico.}
\email{rodrigo.davila@im.unam.mx}

  \thanks{BLANK} 



\begin{abstract}
In the present work we are going to give a formal exposition of the ribbon graphs topic based on notes of Labourie \cite{Lab}, since is difficult to find as such in the literature. As an application we are going to  compute the fundamental group of surfaces using ribbon graphs as a combinatorial version of it.
\end{abstract}

\maketitle

\section*{introduction}
The ribbon graphs gain their mathematical popularity through the work of Penner \cite{Pen} who introduce a cell decomposition of Riemann moduli space, which was later used in Kontsevich's proof of Witten conjecture \cite{Kon}. The ribbon graphs are very useful for the study of the representation variety of surface groups $Hom(\pi_1(S),G)/G$ for a given surface $S$ and a group $G$. In the present work we are going to define the ribbon graphs, then we are going to use them to proof the classification theorem of surface and we are going to compute the fundamental group of a surface using the fundamental group of ribbon graphs.
 
 \section{Surfaces as 2-dimensional manifolds}
 
 \begin{definition}
 A surface is a connected 2-dimensional smooth manifold.
 \end{definition}
 
 A 2-dimensional chart for a surface $S$ is a pair $(U,\phi)$ where $U\subset S$ is an open set and $\phi:U\rightarrow \phi(U)\subset\mathbb{R}^2$ is an homeomorphism on its image.\\
 
 A collection of charts $\{(U_i,\phi_i)\}_{i\in I}$ is called an atlas for $S$ if $S=\cup_{i\in I}U_i$ and we say that the atlas is smooth or $C^{\infty}$ if the change of coordinates $\phi_i\circ\phi_j^{-1}:\phi_j(U_i\cap U_j)\rightarrow\phi_i(U_i\cap U_j)$ is a smooth function for all $i,j\in I$. Given a function $f:S\rightarrow\mathbb{R}$ we say that $f$ is smooth if $f\circ\phi_i^{-1}:\phi_i(U_i)\rightarrow\mathbb{R}$ is a smooth function from $\mathbb{R}^2$ to $\mathbb{R}$.\\

\begin{figure}[h]
\includegraphics{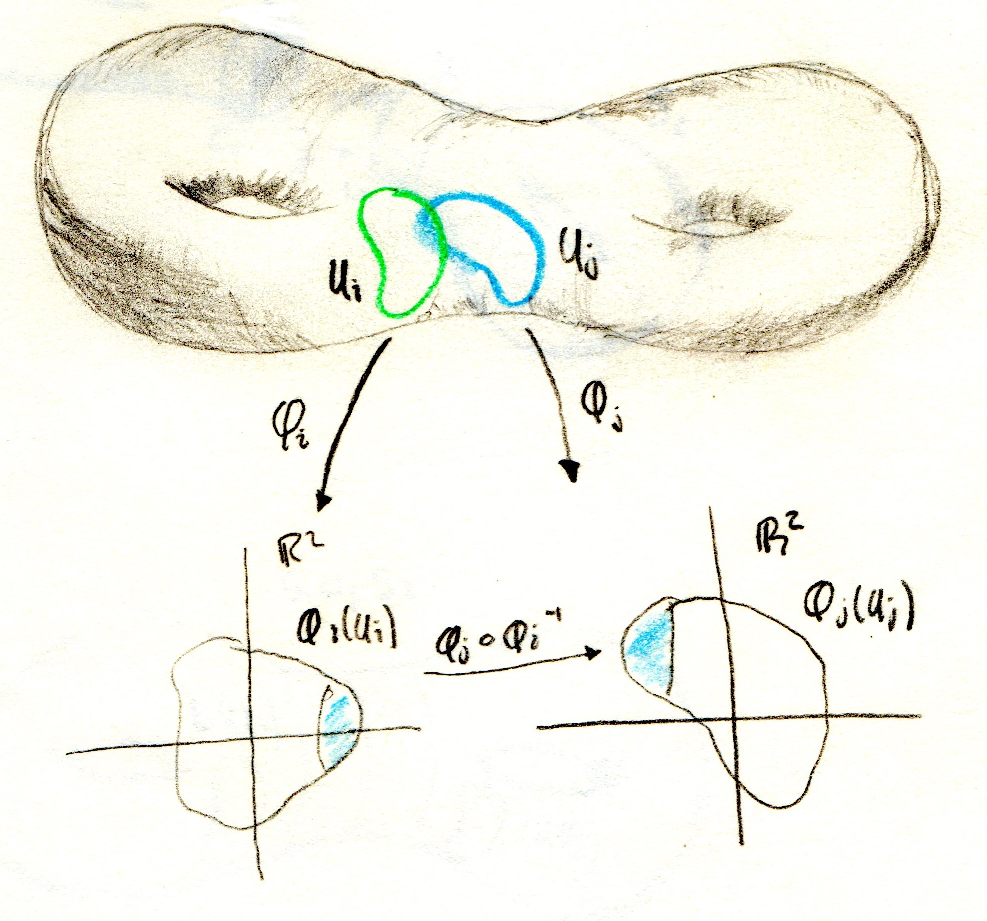}\centering
\caption{A surface.}
\end{figure}

\newpage
 \begin{definition}
 Let $S$ be a surface with atlas $\{(U_i,\phi_i)\}_{i\in I}$. The atlas is called oriented if the jacobian 
 $$Jac(\phi_i, \phi_j):=det(D(\phi_i\circ\phi_j^{-1}))$$
 is positive for all $i,j\in I$. Then we say that S is oriented.  
 \end{definition}

\subsection{Surfaces with boundary}
Let $H^+:=\{(x,y)\in\mathbb{R}^2\mbox{ } |\mbox{ } y\geq0\}$ be the closed upper half plane and $\partial H^+:=\{(x,y)\in\mathbb{R}^2\mbox{ } |\mbox{ } y=0 \}$ its boundary.\\

Given a surface $S$, a two dimensional chart with boundary is a pair $(U,\phi)$ where U is an open subset of $S$ and $\phi:U\rightarrow V\subset H^+$ is an homeomorphism into an open subset $V$ of $H^+$. The subset $\partial U:=\phi^{-1}(\phi(U)\cap\partial H^+)\subset U$ is the boundary of $U$.

\begin{definition}
Let $V_1$, $V_2$ be open subsets of $H^+$. A function $f:V_1\rightarrow V_2$ is smooth if there is an open subset $\tilde{V_1}$ of $\mathbb{R}^2$ with $\tilde{V_1}\cap H^+=V_1$ and a smooth function $\tilde{f}:\tilde{V_1}\rightarrow\mathbb{R}^2$ such that $f=\tilde{f}|_{V_1}$.
\end{definition}

An atlas of charts with boundary $\{(U_i,\phi_i)\}_{i\in I}$ is smooth if the change of coordinates $\phi_i\circ\phi_j^{-1}:\phi_j(U_i\cap U_j)\rightarrow\phi_i(U_i\cap U_j)$ is smooth for all $i,j\in I$ in the sense of the last definition.
\newpage
\begin{definition}
A surface with boundary is a surface $S$ with a smooth atlas of charts with boundary.
\end{definition}

Given a surface $S$ with boundary we say that $x\in S$ is a boundary point if $x\in\partial S$ for any chart $(U,\phi)$ containing it. The set of boundary points of $S$ is denoted by $\partial S$.

\subsection{Gluing surfaces}

We need to know how to construct new surfaces from surfaces with boundary, in order to do this, we need to glue the surfaces along their boundaries and we need to know how they look like in a neighborhood of a boundary component. For this we have to use the following lemma.

\begin{lemma}[Collar Lemma]
Let  $S$ be a surface with boundary $\partial S$ and $\mathcal{C}\subset \partial S$ a connected component. Then there is a neighborhood $U$ of $\mathcal{C}$ in $S$ and a diffeomorphism $\psi:U\rightarrow V$ into a subset $V\subset\mathbb{R}^2$ of the form $V\simeq\mathcal{C}\times[0,1]$ mapping $\mathcal{C}$ into $\mathcal{C}\times\{0\}$.
\end{lemma}

\begin{figure}[h]
\includegraphics{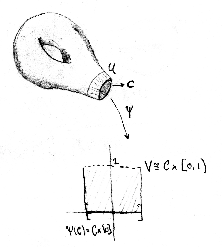}
\centering
\caption{Collar Lemma.}
\end{figure}

Let $S_1$, $S_2$ be two surfaces with boundary. Let $\mathcal{C}_1$ and $\mathcal{C}_2$ be two diffeomorphic components of $\partial S_1$ and $\partial S_2$ respectively. Then the gluing of a surface is as follows: Let $f:\mathcal{C}_1\rightarrow\mathcal{C}_2$ be a diffeomorphism. Consider the disjoint union $S_1\sqcup S_2$ and the following equivalence relation 
$$ x\sim y\Leftrightarrow y=f(x)$$
for $x\in\mathcal{C}_1$ and $y\in\mathcal{C}_2$. Then $S_1\cup_f S_2=S_1\sqcup S_2/\sim$ and we call this quotient the gluing surface of $S_1$ and $S_2$.

\begin{figure}[h]
\includegraphics{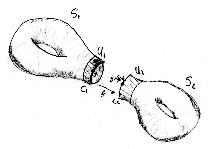}
\centering
\caption{Gluing surfaces.}
\end{figure}

The atlas of $S_1\cup_f S_2$ is now given in the following way: First, we take a smooth atlas $\{(U_i,\phi_i)\}_{i\in I}$ for $S_1\backslash\mathcal{C}_1$ and a smooth atlas $\{(V_j,\psi_j)\}_{j\in J}$ for $S_2\backslash\mathcal{C}_2$. We denote by $\iota_1:S_1\hookrightarrow S_1\cup_fS_2 $, $\iota_2:S_2\hookrightarrow S_1\cup_fS_2$ the canonical inclusions. Then $\{(\iota_1(U_i),\phi_i\circ\iota_1^{-1})\}_{i\in I}\cup\{(\iota_2(V_j),\psi_j\circ\iota_2^{-1}\}_{j\in J}$ is an atlas for the complement of the gluing curve in $S_1\cup_fS_2$.\\

\begin{figure}[H]
\includegraphics[scale=0.65]{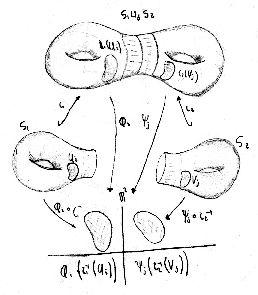}
\centering
\caption{Atlas for the complement of the gluing curve in $S_1\cup_f S_2$.}
\end{figure}

Now we consider a chart for the gluing curve which is compatible with the charts given above.\\
Using the collar lemma we construct this charts. Let $B_1$, $B_2$ be collar neighborhoods of $\mathcal{C}_1$, $\mathcal {C}_2$ respectively, where $\mathcal{C}_i$ are connected components of $\partial S_i$ for $i=1,2$, and $g_1:B_1\rightarrow\mathcal{C}_1\times(-1,0]$ a diffeomorphism such that $g_1(\mathcal{C}_1)=\mathcal{C}_1\times\{0\}$ and $g_2:B_2\rightarrow\mathcal{C}_2\times[0,1)$ a diffeomorphism such that $g_2(\mathcal{C}_2)=\mathcal{C}_2\times\{0\}$. Now consider the open subset $\mathcal{O}:=B_1\cup_fB_2$ of $S_1\cup_fS_2$ and fix an emmbeding $\iota:\mathcal{C}_2\times(-1,1)\rightarrow\mathbb{R}^2$. This is possible since $\mathcal{C}_2$ is either an interval or a circle. We define coordinates for $\mathcal{O}$ by $\psi:\mathcal{O}\rightarrow\mathbb{R}^2$ as 

$$\psi(x)=\left\{\begin{array}{rccl}(\iota\circ(f,id)\circ g_1)(x)& \mbox{if}  & x\in B_1\\ (\iota\circ g_2)(x) & \mbox{if} & x\in B_2\end{array}\right.$$
where $f:\mathcal{C}_1\rightarrow\mathcal{C}_2$ is a diffeomorphism.\\

\begin{figure}[h]
\includegraphics{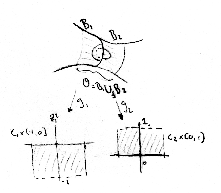}
\centering
\caption{Atlas for the gluing curve.}
\end{figure}

We have proven the following proposition
\begin{proposition}
$S_1\cup_fS_2$ is a surface with smooth atlas given by
$$\{(\iota_1(U_i),\phi_i\circ\iota_1^{-1})\}_{i\in I}\cup\{(\iota_2(V_j),\psi_j\circ\iota_2^{-1})\}_{j\in J}\cup\{\mathcal{O},\psi\}.$$
\end{proposition}

\begin{bf}Remark:\end{bf}
Given two oriented surfaces with boundary $S_1$, $S_2$ we can give a unique orientation to $S_1\cup_fS_2$ compatible with the orientation of $S_1$ and $S_2$ using an orientation reversing diffeomorphism\\ $f:\mathcal{C}_1\rightarrow\mathcal{C}_2$, where $\mathcal{C}_1$ and $\mathcal{C}_2$ are connected components of $\partial S_1$ and $\partial S_2$ respectively. 

\section{Surfaces as combinatorial objects}
\subsection{Ribbon graphs}
In an informal way, a graph is a collection of points called vertex which are joined by some lines called edges as in the Figure 6. If we choose an orientation on the edges, we say that the graph is oriented or directed. The following definition gives us a formal description of these objects.
\begin{definition}
An oriented graph $\Gamma$ is a triple $\Gamma=(V,E,i)$, where $V$ is a finite set $V=\{v_1,\ldots,v_n\}$ whose elements are called vertex and $E$ is a finite set whose elements are called edges and a map $i:E\rightarrow V\times V$ with $i(e)=(e_-,e_+)$, where $e_-$ is the origin of the edge $e$ and $e_+$ is the end of the edge $e$.\\

\begin{figure}[h]
\includegraphics{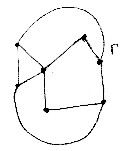}
\centering
\caption{Example of a graph $\Gamma$.}
\end{figure}

We say that an edge and a vertex are incident if the vertex is on the image of the edge under the map $i$. The quantity $a_{jk}=|i^{-1}(v_j,v_k)|$ gives us the number of edges that connect two vertex $v_j$ and $v_k$.\\

The degree or valence of a vertex $v_j$ is the number given by
$$deg(v_j)=\sum_{k\neq j} a_{jk}+2a_{jj}$$
which is the number of edges incident to $v_j$. A loop, this is an edge whit just one incident vertex, contributes twice to the degree.  
\end{definition}
 
\begin{definition}
The edge refinement of an oriented graph $\Gamma=(V,E,i)$ is the graph $\Gamma_E=(V\sqcup V_E, E\sqcup E, i_E)$ with a point added at each edge as a degree 2 vertex, where $V_E$ denotes the set of this vertices. The set of vertices of $\Gamma_E$ is $V\sqcup V_E$ and the set of edges is $E\sqcup E$. The incidence relation is described by the map $i_E: E\sqcup E\rightarrow V\times V_E$ because each edge of $\Gamma_E$ connects exactly one vertex of V to a vertex of $V_E$ and an edge of $\Gamma_E$ is called a half-edge.
\end{definition}

For each vertex $v\in V$ of $\Gamma_E$ the set $i_E^{-1}(\{v\}\times V_E)$ consists of half-edges incident to $v$ and we have $deg(v)=|i_E^{-1}(\{v\}\times V_E)|$.\\

\begin{figure}[h]
\includegraphics{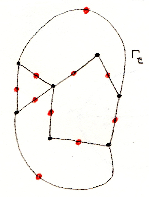}
\centering
\caption{Edge refinement $\Gamma_E$ of the graph $\Gamma$.}
\end{figure}

\begin{bf}Remark:\end{bf}
Let $e\in E$  be an edge of $\Gamma$, then $i(e)=(e_-,e_+)$. We denote by $e_0$ the vertex added on the edge $e$ in the refinement of $\Gamma$, i.e, $e_0\in V_E$ and we denote by $e^- $ and $e^+$ in $E\sqcup E$  the edges such that $i_E(e^-)=(e_-,e_0)$ and $i_E(e^+)=(e_0,e_+)$.\\

Let $\Gamma=(V, E, i)$ be an oriented graph and let $I:E\rightarrow E$ be an involution map with $I(e)=\bar{e}$ where $\overline{e_-}=e_+$ and $\overline{e_+}=e_-$. We call the pair $(e,\bar{e})\in E\times E$ a geometric edge of the graph $(\Gamma,I)$.\\

The geometric realization $|\Gamma|_I$ of the graph $(\Gamma,I)$ is a topological space $|\Gamma|_I=E\times[0,1]/\sim$ where $\sim$ is the equivalence relation generated by the relations
\begin{itemize}
\item $(e,t)\sim(\bar{e},1-t)$
\item If $e,f\in E$ with $e_-=f_-$ then $(e,0)\sim(f,0)$
\item If $e,f\in E$ with $e_+=f_+$ then $(e,1)\sim(f,1)$
\end{itemize}

\begin{bf} Remark:\end{bf}
The geometric realization of a graph $\Gamma=(V,E,i)$ not always can be drawn on a plane $\mathbb{R}^2$ without intersections, however, we can draw the geometric realization of a graph on $\mathbb{R}^3$. To do this let $p:\mathbb{R}\rightarrow\mathbb{R}^3$ be the function $p(t)=(t,t^2,t^3)$ and $C$ be the curve $C=\{p(t):t\in\mathbb{R}\}$. Now we only need to take any vertex $v_i\in V$ into the curve, and to see that the edges do not intersect we need only show that given four vertex on $C$ they are not coplanar. Now for any four points $v_1, v_2, v_3,v_4$ in $\mathbb{R}$, the volume of the tetrahedron $T$ formed by $p(v_i)\in C$ is proportional to a Vandermonde determinant:
$$6Vol(T)=det((p(v_2)-p(v_1))\cdot [(p(v_3)-p(v_1))\times(p(v_4)-p(v_1))])=det\left[\begin{array}{rccl}1&v_2&v_1^2&v_1^3\\ 1&v_2&v_2^2&v_2^3\\ 1&v_3&v_3^2&v_3^3\\ 1&v_4&v_4^2&v_4^3\end{array}\right]\neq0$$
this implies that any four points on $C$ are not coplanar. As a result, the edges of the tetrahedron $T$ intersect only the appropriate vertex. Now we take arbitrary $n$ distinct points $p_i$ in $C$. The argument above show that if we form the graph $\Gamma$ from this $n$ points, the edges intersect only in the appropriate vertex and this gives us an embedding of the given graph $\Gamma$ into $\mathbb{R}^3$.\\
Now using this fact we can project the graph into the plane in a such way that the edges cross over or under as in the figure 8.\\

\begin{figure}[h]
\includegraphics{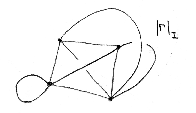}
\centering
\caption{Geometric realization of a graph whit under and over crossings.}
\end{figure}

In the same way we can make the geometric realization of the refinament of $\Gamma$.\\

Now, let's we define morphisms between graphs.

\begin{definition}
A traditional graph isomorphism $\phi=(\alpha,\beta)$ between two graphs $\Gamma_1=(V_1,E_1,i_1)$ and $\Gamma_2=(V_2,E_2,i_2)$ is a pair of bijective maps 
$$\alpha:V_1\rightarrow V_2,\ \beta:E_1\rightarrow E_2$$
 that preserves the incidence relation, i.e., the following diagram commutes 
 $$\xymatrix{ E_1 \ar[d]_\beta \ar[r]^{i_1} & V_1\times V_1 \ar[d]^{\alpha\times\alpha}\\ E_2\ar[r]_{i_2} &V_2\times V_2}$$
\end{definition}
\begin{theorem}
Let $(\Gamma_1, I)$ and $(\Gamma_2, I)$ be isomorphic graphs, then $|\Gamma_1|_I$ and $|\Gamma_2|_I$ are homeomorphics. 
\end{theorem}
\begin{proof}
Let $\phi=(\alpha,\beta):\Gamma_1\rightarrow\Gamma_2$  be an isomorphism of graphs with $\beta(e)=f$ or $\beta(e)=\bar{f}$.\\
Now, consider $E_j$ with $j=1,2$ with the discrete topology, since $\phi$ is an isomorphism we have that $\beta:E_1\rightarrow E_2$ is an homeomorphism and we define the function
$$\beta\times id:E_1\times[0,1]\rightarrow E_2\times[0,1]$$
and we have the quotient maps $q_j:E_j\times[0,1]\rightarrow|\Gamma_j|_I$ with $j=1,2$. To see that this maps induces an homeomorphism on the geometric realization we just need to show that the following diagram commutes and the functions are continuous:
$$\xymatrix{ E_1\times[0,1] \ar[d]_{q_1}\ar[r]^{\beta\times id} & E_2\times[0,1]\ar[d]^{q_2}\\ |\Gamma_1|_I\ar@{-->}[r]& |\Gamma_2|_I}$$ 

To do this we define the function $|\phi|:|\Gamma_1|_I\rightarrow|\Gamma_2|_I$ which maps $[(e,t)]\mapsto[(\beta(e),t)]$ where $[\mbox{ }]$ denotes the equivalence class. Let's see that $|\phi|$ is well defined:
If we have $[(e,t)]$ then $(e,t)\sim(\bar{e},1-t)$ takes $(\bar{e},1-t)$. Now since $\beta(e)=f$ or $\beta(e)=\bar{f}$ then $\beta(\bar{e})=\bar{f}$ or $\beta(\bar{e})=f$. Suppose that $\beta(e)=f$, the other case is similar, then $\beta(\bar{e})=\bar{f}$ but $(\bar{f},1-t)\sim(f,t)=(\beta(e),t)$ therefore $[(\beta(\bar{e}),1-t)]=[(\beta(e),t)]$. Now since the diagram in the definition 2.3 commutes we have that $[(e,0)]\mapsto [(\beta(e),0)]$ and $[(e,1)]\mapsto[(\beta(e),1)]$, therefore the map is well defined and the diagram commutes.
Let's see that the function $|\phi|$ is continuous.
Let $U\subset|\Gamma_2|_I$ be an open subset, since $q_2$ is a quotient map, therefore continuous, we have that $q_2^{-1}(U)$ is open in $E_2\times[0,1]$ and $\beta\times I$ is continuous, then $(\beta\times id)^{-1}(q_2^{-1}(U))$ is an open subset of $E_1\times[0,1]$ and we have that $q_1$ is a quotient map, therefore an open map, then $q_1((\beta\times[0,1])^{-1}(q_2^{-1}(U)))$ is open in $|\Gamma_1|_I$ therefore $|\phi|$ is continuous and since $\beta\times id$ is an homeomorphism then $|\phi|$ is an homeomorphism.
\end{proof}

Now, let's consider graphs with more structure. To do this we need the notion of cyclic ordering on a finite set $S$.
\begin{definition}
A cyclic ordering in a finite set $S$ is a bijection $\sigma: S\rightarrow S$ such that for all $x\in S$ the orbit $\{\sigma^n(x)\}_{n\in\mathbb{Z}}=S$. Given $x\in S$ we will call $\sigma(x)$ the successor of $x$ and $\sigma^{-1}(x)$ the predecessor of $x$.
\end{definition} 

\begin{definition}[Ribbon Graph]
Let $(\Gamma,I)$ be a graph. For $v\in V$ the star of $v$ 
$$E_v=\{e\in E\ :\  v=e_-\}$$
is the set of edges starting from $v$. A ribbon graph is the graph $(\Gamma,I)$, together a cyclic ordering on the star of every vertex.
\end{definition}

\begin{figure}[h]
\includegraphics{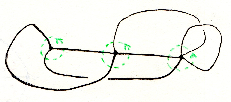}
\centering
\caption{Ribbon graph.}
\end{figure}

\begin{bf}Remark:\end{bf}\\
\begin{enumerate}
\item We can see the star of a vertex $v\in V$ as the set of semi-edges starting on $v$ when we consider the refinement of the graph and we denote this set as $E_v^*$.\\

\item We can consider isomorphisms of ribbon graphs. We just need to ask that preserve the cyclic ordering on each star, this is, that the following diagram commutes 
$$\xymatrix{ E_v^1 \ar[d]_{\sigma_v^1}\ar[r]^\beta&E_w^2\ar[d]^{\sigma_w^2} \\E_v^1\ar[r]_\beta& E_w^2}$$
where $\beta: E\rightarrow E$ is a bijection from the edges of the graphs 
$$\phi=(\alpha,\beta):\Gamma_1=(V_1,E_1,i_1,I)\rightarrow\Gamma_2=(V_2,E_2,i_2,I)$$ and $\sigma_v^1$ is the cyclic ordering on $E_v^1$ and $\sigma_w^2$ is the cyclic ordering on $E_w^2$.\\
This isomorphism induces an homeomorphism on the geometric realization of the ribbon graph preserving the cyclic ordering on it.
\end{enumerate}

For all $v\in V$ we can consider an embedding on $\mathbb{R}^2$ of the geometric realization of $E_v^*$, the orientation of $\mathbb{R}^2$ induces a cyclic ordering on each star in the following way: let's consider a circle with center on $v$ and radius one (we can suppose this because we can consider the length of each edge on $|E_v^*|_I$ as 1), since the circle gets an orientation from $\mathbb{R}^2$ and we can define the cyclic ordering from this orientation. (see figure 9).\\

Lets consider surfaces from the ribbon graphs. In order to do this we need first to embed the geometric realization of the graph in a open oriented surface (this is not always possible to do in the plane).

\begin{lemma}
Every ribbon graph can be embedded in an open oriented surface such that its cyclic ordering are induced from the orientation of the surface. 
\end{lemma} 
\begin{proof}
We construct the surface in the following way:
Let $v\in V$ be a vertex of ($\Gamma$,I) and $|E_v^*|_I$ the geometric realization of the star at $v$ in the refinement of ($\Gamma$,I), consider an embedding of $|E_v^*|_I$ in $\mathbb{R}^2$ and take a disc $D(v)$ with center on $v$ and radius 1 (we can consider this length for each edge in $|E_v^*|_I$). Now we take a tubular neighborhood $U_v\subset D(v)$ of $|E_v^*|_I$ with many boundary components as elements in $E_v^*$ labelled in the following way by the elements of $E_v$. Start with an arbitrary edge  $e\in|E_v^*|_I$, then the following boundary component is labelled by $\sigma_v(e)$ where $\sigma_v$ is the cyclic ordering of $|E_v^*|_I$ and so on.\\

\begin{figure}[h]
\includegraphics{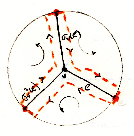}
\centering
\caption{Tubular neighborhood of the star and its orientation.}
\end{figure}

Since we do this for all vertex $v\in V$, using the gluing lemma, we now glue  each star with the other ones in the following way: for $e\in|E_v^*|_I$ and $e'\in|E_v'^*|_I$ we glue their boundary components if we have that $i_E(e)=(v,e_{0})$ and $i_E(e')=(e_{0},v')$, this is, that $e=e^+$ and $e'=e^-$ in the refinement of $(\Gamma, I)$ (see figure 11), and we make this gluing such that the orientations are reversed, then we get an oriented open surface. Since in each vertex the ordering of its corresponding star is preserved for the process of gluing, then the orientation of the surface is compatible with $\sigma_v$.
\end{proof}

\begin{figure}[h]
\includegraphics{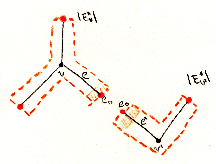}
\centering
\caption{Gluing of the stars.}
\end{figure}

The surface constructed in the proof of Lemma 2.7 is called the associated ribbon surface of the graph.\\

We can associate different ribbon graphs to a one graph all depends on the cyclic ordering in the edges (see figure 12) and we get different associated ribbon surfaces (see figure 13).

\begin{figure}[h]
\subfloat{\includegraphics{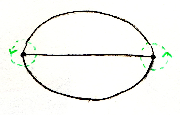}}
\subfloat{\includegraphics{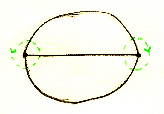}}
\centering
\caption{Ribbon graphs $\Gamma_1$ and $\Gamma_2$ from the same graph $\Gamma$.}
\end{figure}

In order to embed ribbon graphs into closed surfaces we need to close the holes in the associated ribbon surface. To do this we define the faces of the associated ribbon surface.
\begin{definition}
Let $(\Gamma=(V, E,i),I)$ be a ribbon graph. A face is an n-tuple $(e_1,\ldots,e_n)$ of edges such that $e_p^+=e_{p+1\mod n}^-$ and $\sigma_{e_p^+}(\bar{e_p})=e_{p+1\mod n}$ for all $1\leq p\leq n$ where $\sigma_{e_p^+}$ is the cyclic ordering on the star of $e_p^+$
\end{definition}

\begin{figure}[h]
\subfloat{\includegraphics{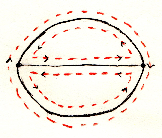}}
\subfloat{\includegraphics{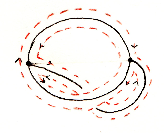}}
\centering
\caption{Associated ribbon surfaces of $\Gamma_1$ and $\Gamma_2$.}
\end{figure}

 The boundaries of this faces will be the boundaries of the disc we will be attaching at our associated ribbon surface.
 \begin{definition}
 A graph $\Gamma$ embedded in a surface $S$ is filling if each connected component of $S\backslash |\Gamma|_I$ is diffeomorphic to a disc.
 \end{definition} 
 
 Now we have:
 \begin{proposition}
 Every ribbon graph $(\Gamma,I)$ has a filling embedding into a compact oriented surface $S$. The connected components of $S\backslash|\Gamma|_I$ are in bijection with the faces of the associated ribbon surface of $\Gamma$.
 \end{proposition}
\begin{proof}
 Since $\Gamma$ is a ribbon graph we have that the boundary components of the associated ribbon surface of $\Gamma$ define closed curves homeomorphic to a circles. We glue a disc for each of these curves. Therefore we thus obtain a closed surface and is followed immediately that the connected components of $S\backslash|\Gamma|_I$ are in bijection with the faces of the associated ribbon surface.$\blacksquare$
 \end{proof}

 We will see that the surface obtained from the proposition 2.10 is unique in a very strong sense. To see this, we will need the following basic fact from point-set topology.\\
 
 \begin{lemma}[Clutching Lemma]
 Let $X=U\cup V$ be a decomposition of a topological space $X$ in two closed sets $U$ and $V$. If $f_1:U\rightarrow Y$ and $f_2:V\rightarrow Y$ are continuous maps from $U$ and $V$ into some topological space $Y$ such that $f_1|_{U\cap V}=f_2|_{U\cap V}$ then the induced map $f:X\rightarrow Y$ is continuous.
 \end{lemma}
 
 Using this we can show the following result:
 
 \begin{proposition}
 Let $\Gamma_1\subset S_1$ and $\Gamma_2\subset S_2$ be filling ribbon graphs of compact oriented surfaces and let $\phi:\Gamma_1\rightarrow\Gamma_2$ be an isomorphism of ribbon graphs. Then $\phi$ induces an homeomorphism on the geometric realization $|\phi|:|\Gamma_1|_I\rightarrow|\Gamma_2|_I$ and this extends to an homeomorphism between $S_1$ and $S_2$.
 \end{proposition}
\begin{proof}
Since $\phi$ is an isomorphism of ribbon graphs, then by Theorem 2.4 this extends to an homeomorphism of the geometric realization $|\phi|:|\Gamma_1|_I\rightarrow|\Gamma_2|_I$. Let $S_{\Gamma_1}$ and $S_{\Gamma_2}$ be the associated ribbon surfaces of $\Gamma_1$ and $\Gamma_2$ respectively, then by the clutching lemma this homeomorphism extends to an homeomorphism of the closure of the associated ribbon surfaces.\\

Since $\Gamma_1$ and $\Gamma_2$ are filling we have that 
$$S_1\backslash|\Gamma_1|_I=\sqcup_{f\in F}D_f,\ S_2\backslash|\Gamma_2|_I=\sqcup_{g\in G}D'_{g}$$
where $D_f$ and $D'_g$ are discs. Then we have that
$$S_1=\overline{S_{\Gamma_1}}\cup(\sqcup_{f\in F}\bar{d_f}),\ S_2=\overline{S_{\Gamma_2}}\cup(\sqcup_{g\in G}\bar{d'_g})$$
where $d_f$ and $d'_g$ are slightly smaller discs than $D_f$ and $D'_g$ and
$$\overline{S_{\Gamma_1}}\cap(\sqcup_{f\in F}\bar{d_f})=\sqcup_{f\in F}\partial\bar{d_f}, \overline{S_{\Gamma_2}}\cap(\sqcup_{g\in G}\bar{d'_g})=\sqcup_{g\in G}\partial\bar{d'_g}$$
are the unions of discs. By the Clutching lemma is suffices to construct for each $f\in F$ an homeomorphism from $d_f$ to $d'_g$ which agrees on the boundary $\partial\bar{d_f}$ with the extension $|\phi|$ to $\overline{S_{\Gamma_1}}$. But if $\psi:\mathbb{S}^1\rightarrow\mathbb{S}^1$ is any homeomorphism of circles then there is an obvious way to extend it to the corresponding disc. In fact, each $x$ in the disc may be written in polar coordinates as $x=re^{i\theta}$ for some $r\in[0,1]$ and some $e^{i\theta}\in\mathbb{S}^1$. Then we can simply define $\psi(re^{i\theta})=r|\phi|(e^{i\theta})$ to obtain the desired homeomorphism.
\end{proof}

Combining proposition 2.10 and 2.12 we have
\begin{corollary}
For any ribbon graph $\Gamma$ there exists a unique compact oriented surface $S_\Gamma$ (up to homeomorphism) such that $\Gamma$ can be embedded as a filling ribbon graph into $S_\Gamma$.
\end{corollary}
Corollary 2.13 will enable us to classify surfaces up to homeomorphism and allows us to construct surfaces from ribbon graphs. The following proposition is a converse of this corollary.

 \begin{figure}[h]
\subfloat{\includegraphics{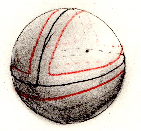}}
\subfloat{\includegraphics{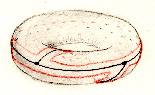}}
\centering
\caption{Closed surfaces from the ribbon graphs $\Gamma_1$ and $\Gamma_2$.}
\end{figure}

\begin{proposition}
Every compact oriented surface admits a filling ribbon graph.
\end{proposition}
The proof of this proposition is similar to find a triangulation of the surface. Then by a theorem of Cairns and Whitehead we have that every smooth manifold admits a triangulation, since every surface is a two dimensional smooth manifold we are done (See \cite{Cai} and \cite{Whi}).


\section{Classification of surfaces I: Existence}
By corollary 2.13 a convenient description of the compact oriented surfaces is given by their underlaying filling ribbon graph. We consider a family $\Gamma_g$, $g\geq 1$ of filling ribbon graphs given in the following way: Take g copies of the graph $\Gamma_1$ shown in the Figure 15, we call this graph a petal.

\begin{figure}[H]
\includegraphics{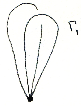}
\centering
\caption{A petal $\Gamma_1$}
\end{figure}

Then we glue g-copies of the petal by they vertex to get a graph $\Gamma_g$ with $g$ petals as shown in the following figure.\\ 

\begin{figure}[h]
\includegraphics{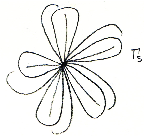}
\centering
\caption{$5$-petals ribbon graph $\Gamma_5$}
\end{figure}

Using definition 1.10 we can se that $\Gamma_1$ has one face and with some mental gymnastics we can see that the associated oriented closed surface $S_1:=S_{\Gamma_1}$ is a torus.\\

\begin{figure}[H]
\includegraphics{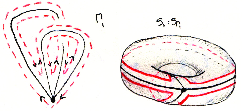}
\centering
\caption{Surface $S_1$}
\end{figure}

Similarly, each copy of $\Gamma_1$ in $\Gamma_g$ is a torus with one puncture and we glue two consecutive torus by their punctures. Thus, we get a surface $S_g:=S_{\Gamma_g}$ which is a handlebody with $g$ handles.\\

There is another useful description of $S_g$ as follows: Let $D_g=S_g\backslash |\Gamma_g|$ be the disc in $S_g$ which corresponds to the only face of $\Gamma_g$. Then $S_g$ is obtained by gluing the boundary of the disc $D_g$ in the following way: Let $a_i$, $b_i$ be two edges of the $i$-th copy of $\Gamma_1$ in $\Gamma_g$. Since each oriented edge of $\Gamma_g$, this is $a_i,\overline{a_i}, b_i,\overline{b_i}$, occurs only once in the boundary of $D_g$, then we can describe the boundary of $D_g$ by the series of edges given by
$$a_1,b_1,\overline{a_1},\overline{b_1},\ldots, a_g,b_g,\overline{a_g},\overline{b_g}$$

We can see this for the case $g=2$ as shown in the Figure 18

\begin{figure}[h]
\includegraphics{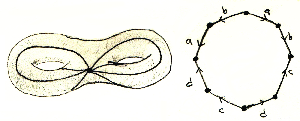}
\centering
\caption{The surface $S_2$ and its gluing polygon}
\end{figure}

It is convenient to define $S_0:= \mathbb{S}^2$ the two-sphere.\\

Later we see that given any filling ribbon graph we can deform it to any of the $\Gamma_g$ graphs. Now we can state the first part of the classification theorem.
 
\begin{theorem}
Every oriented compact surface $S$ is homeomorphic to one of the surfaces $S_g$ for $g\geq0$
\end{theorem}
\begin{proof}
 By proposition 2.14 we can choose a filling ribbon graph $\Gamma$ for the surface $S$. If $\Gamma$ doesn't have edges, then $S$ must be the surface $S_0$. Thus we may assume that $\Gamma$ has at least one edge. Now, we will deform the graph $\Gamma$ to obtain one of the graphs $\Gamma_g$ without changing the filling property in the process. In this way the theorem follows from corollary 2.13.\\
 
 First we deform the graph $\Gamma$ so that we see that the surface is obtained from gluing the boundary of a polygon, in the following way:
 \begin{enumerate}
 \item Eliminating faces: Let's assume that $\Gamma$ has more than one face. Then, there is a geometric edge $(e,\bar{e})$ such that $e$ and $\bar{e}$ are in different faces. Let $\Gamma'$ be the graph obtained by eliminating from $\Gamma$ the edges $e$ and $\bar{e}$, this is $\Gamma'=\Gamma\backslash\{e,\bar{e}\}$, then $\Gamma'$ is still filling and has one face less than $\Gamma$ as shown in the Figure 19.
 
 \begin{figure}[h]
\includegraphics{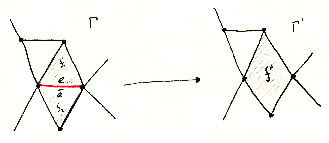}
\centering
\caption{Eliminating edges}
\end{figure}
 
If we iterate this process we get a filling ribbon graph with only one face.
\item Eliminating Vertices: Let $\Gamma=(V,E,I)$ be a filling ribbon graph and $\phi:|\Gamma|_I\hookrightarrow S$ be an embedding on the surface $S$. If $\Gamma$ has more than one vertex, lets say $e_-$ and $e_+$, then there is an edge $e$ joining them. Let $\Gamma'=(V',E',I)$ be a new filling ribbon graph and $\phi':|\Gamma'|_I\hookrightarrow S$ an embedding in $S$ where:
\begin{itemize}
\item The new set of vertices is obtained by crushing the vertex $e_-$ and $e_+$ in a single vertex $e_c$, thus $V'=(V\backslash\{e_-,e_+\})\cup\{e_c\}$.
\item The new set of edges is given by $E'=E\backslash\{e,\bar{e}\}$.
\item The map $\phi':|\Gamma'|\hookrightarrow S$ is defined by sending the new vertex into a point on the geometric image of $\phi(e)$ and is extended to the edges which previously started from $e_{\mp}$.
\end{itemize}

Again this process does not change the filling property and reduce the number of vertices by one without increasing the number of faces as in the following figure.

\begin{figure}[h]
\includegraphics{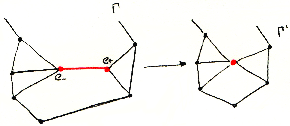}
\centering
\caption{Eliminating vertices}
\end{figure}
Iterating this process we obtain a graph with only one vertex.
 \end{enumerate}
Then by (1) and (2) we can assume that the graph $\Gamma$ has only one vertex and one face. Therefore we get that $S$ is obtained by gluing the sides of a polygon labelled by the edges of the graph and by definition of face, every oriented edge appears once, then the gluing is given by identifying $e$ and $\bar{e}$ with reversed orientation.\\

If there are no edges left, then $S$ is the surface $S_0$ and we are done. Thus, assume that $\Gamma$ has at least one edge. Let $(a,\bar{a})$ and $(b,\bar{b})$ geometric edges from $\Gamma$. We will call the pair $((a,\bar{a}),(b,\bar{b}))$ linked if their relative position is as in the following figure.

\begin{figure}[h]
\includegraphics{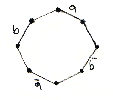}
\centering
\caption{Linked geometric edges}
\end{figure}
\begin{bf}Claim\end{bf}: Any geometric edge of $\Gamma$ is linked to at least other geometric edge.\\
Proof: Assume that $(a,\bar{a})$ is not linked to any other edge, then this edge would produce an additional face since $\Gamma$ is a Ribbon graph, but this contradicts the assumption that there is only one face.\\

The following claim let us rearrenge the labelling of the sides of the polygon in such way that we obtain a graph $\Gamma_g$.\\

\begin{bf}Claim:\end{bf} Given a linked pair $(a,\bar{a})$, $(b,\bar{b})$ of geometric edges, there is a way of rearrenging the labelling of the polygon without changing the resulting quotient space such that
\begin{itemize}
\item $a,b,\bar{a},\bar{b}$ appears as a subsequence of the sides of the polygon.
\item no subsequence of type $c,d,\bar{c},\bar{d}$ is destroyed during this process.
\end{itemize}        
Proof: First we add an edge to obtain two faces as is shown in the figure 22.
\begin{figure}[H]
\includegraphics{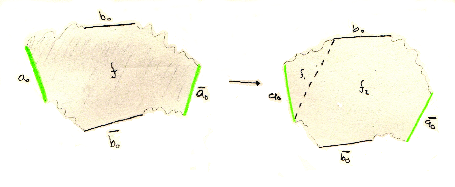}
\centering
\caption{Adding an edge}
\end{figure}
Then we erase in the graph the green edges, which has the effect on the polygon to glue together these two green lines in the red one (see the figure 23).
 \begin{figure}[h]
\includegraphics{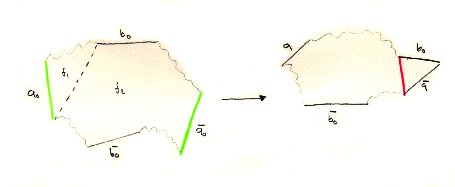}
\centering
\caption{eliminating the green edges}
\end{figure}

\newpage Then we repeat this procedure two more times as is depicted in the figure 24.
In the final picture we have created an additional subsequence of the form $a,b,\bar{a},\bar{b}$ which proves the claim.
\begin{figure}[H]
\includegraphics{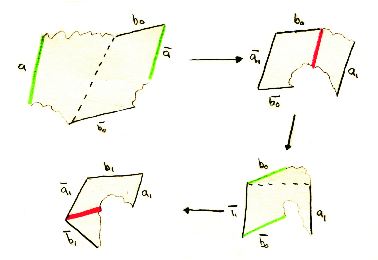}
\centering
\caption{eliminating the green edges}
\end{figure}

Then the resulting surface (which is homeomorphic to $S$) is thus brought into a new position such that all the edges of its polygon are of the form 
$$a_1,b_1,\overline{a_1},\overline{b_1},\ldots,a_g,b_g,\overline{a_g},\overline{b_g}$$
and we conclude the proof of the theorem.  
\end{proof}

\section{The fundamental group of a surface}

In this section, for the proofs of the results we refer the lector to the books \cite{Hat} and \cite{Agu}.

\subsection{The fundamental group of a topological space}

At this moment we have proved half of the classification theorem, in order to prove the other half, we need to know how to distinguish two surfaces $S_g$ and $S_{g'}$ when $g\neq g'$. In order to show this, we need an invariant that distinguishes the surfaces $S_g$  and $S_{g'}$ from each other.\\

This invariant is the fundamental group and we briefly recall its definition:\\
\begin{definition}
Let $X,Y$ be topological spaces
\begin{itemize}
\item A parametrised loop in $X$ based at $x_0\in X$ is a continuous map $\gamma:[0,1]\rightarrow X$ with $\gamma(0)=x_0=\gamma(1)$. We denote by $\Omega(X,x_0)$ the set of loops in $X$ based at $x_0$.
\item The composition of two based loops $\gamma_0, \gamma_1\in\Omega(X, x_0)$ is defined as:

$$(\gamma_0\ast\gamma_1)(t)=\left\{\begin{array}{rccl} \gamma_0(2t), & 0\leq t\leq \frac{1}{2}\\
\gamma_1(2t-1), & \frac{1}{2}\leq t\leq 1\end{array}\right.$$

\item Let $f_0, f_1:Y\rightarrow X$ be continuous maps which agrees on a subset $A\subset Y$. Then $f_0$ and $f_1$ are called homotopic relative to $A$, denoted by $f_0\simeq f_1 (rel A)$, if there exists a map $H:[0,1]\times Y\rightarrow X$ with
$$H(0,y)=f_0(y)$$
$$H(1,y)=f_1(y)$$
$$\forall a\in A,\ H(t,a)=f_0(a)=f_1(a)$$ 
the map $H$ is called homotopy relative to $A$. A space $X$ is called contractible if the identity map $id: X\rightarrow X$ is homotopic to a constant map $x\mapsto x_0$ for some $x_0\in X$.
\item Two based loops $\gamma_0,\gamma_1\in\Omega(X,x_0)$ are called homotopic if there is a homotopy relative to $\{0,1\}$. The set $\pi_1(X,x_0)=\Omega(X,x_0)/\sim$, where $\sim$ is the equivalence relation given by $\gamma_0\sim\gamma_1\Leftrightarrow\gamma_0\simeq\gamma_1(rel\{0,1\})$.
\end{itemize}
\end{definition}

\begin{theorem}
The set $\pi_1(X,x_0)$ is a group with the operation given by $\ast:\pi_1(X,x_0)\times\pi_1(X,x_0)\rightarrow\pi_1(X,x_0)$, $[\gamma_0]\ast[\gamma_1]=[\gamma_0\ast\gamma_1]$ where $[\mbox{ }]$ denotes a homotopy class and $\ast$ is the composition of loops. 
\end{theorem}

\begin{definition}
The group $(\pi_1(X,x_0),\ast)$ is called the fundamental group of the space $X$ based on $x_0$.
\end{definition}
Let $x_0,x_1\in X$ and let $\gamma:[0,1]\rightarrow X$ with $\gamma(0)=x_0$, $\gamma(1)=x_1$. Then the map 
$$\pi_1(X,x_0)\rightarrow\pi_1(X,x_1), [\alpha]\mapsto[\gamma\ast\alpha\ast\bar{\gamma}] $$
is an isomorphism, where the composition of paths is defined as composition of loops above and $\overline{\gamma}(t):=\gamma(1-t)$. The isomorphism type of the fundamental group of an arcwise connected space $X$ not depends on the base point.
\begin{definition}

An arcwise connected space $X$ is called simply-connected if $\pi_1(X,x_0)$ is trivial for some (and hence any) $x_0\in X$.
\end{definition} 
\begin{bf}Examples:\end{bf}
\begin{enumerate}
\item $\pi_1(x_0)=\{1\}$ where $x_0$ is a point.
\item $\pi_1(\mathbb{S}^1)=\mathbb{Z}$
\item $\pi_1(\mathbb{S}^1\times\mathbb{R})=\mathbb{R}$
\item $\pi_1(\mathbb{S}^n)=\{1\}$ for $n\geq 2$
\item $\pi_1(\infty)=F_2$, the free group on two generators
\end{enumerate}

\subsection{Coverings and the fundamental group}

Computing the fundamental group using only the definition is in many cases impossible. One common way to compute the fundamental group is by looking the space as a quotient of a simply-connected space. To do this we need the following notions.

\begin{definition}[Group action]
Let $G$ be a group and $S$ a nonempty set. Then $G$ is said to act on $S$ if there is function from $G\times S$ to $S$, usually denoted $(g,s)\mapsto gs$, such that for the identity $e\in G$, $es=s$ for all $s\in S$, and for all $g,h\in G$ and $s\in S$, $(gh)s=g(hs)$.
\end{definition}
\begin{bf}Remark\end{bf} The previous definition is for left actions, we can define right actions as follows: $S\times G\rightarrow S$ with $(s,g)\mapsto sg$ and with the same properties.\\

\begin{definition}
Suppose that $G$ is a group which acts on a set $S$. If $s\in S$, let $G(s)=\{gs | g\in G\}$. The set $G(s)$ is called the orbit of $s$. The stabilizer of $s$ is the subset $G_s=\{g\in G | gs=s\}$
\end{definition}

\begin{definition}
Let $\Gamma$ be a discrete group which acts on a space $M$. Then the action is called free if it has no fixed points, in other words the stabilizer $\Gamma_x$ is trivial for all $x\in M$. The action is properly discontinuous  if for any compact set $K\subset M$ the set
$$\Gamma_K=\{\gamma\in\Gamma | \gamma K\cap K\neq\emptyset\}$$
is finite.
\end{definition}
The reason of why we need these tools is the following 

\begin{theorem}
Let $\Gamma$ act on a space $X$ proper discontinuously. Then $X$ is Housdorff if and only if $X/\Gamma$ is Housdorff. 
\end{theorem} 
Now we introduce some basic notions of covering spaces.

\begin{definition}
A covering space of a space $X$ is a space $\tilde{X}$ together with a map $p:\tilde{X}\rightarrow X$ such there is an open cover $\{U_\alpha\}_{\alpha\in I}$ of $X$ such that for each $\alpha$, $p^{-1}(U_\alpha)$ is a disjoint union of open sets in $\tilde{X}$, each of which is mapped by $p$ homeomorphically onto $U_\alpha$	
\end{definition}

\begin{definition}
Given a covering $p:\tilde{X}\rightarrow X$, a lifting of a map $f:Y\rightarrow X$ is a map $\tilde{f}:Y\rightarrow \tilde{X}$ such that $f=p\circ\tilde{f}$.
\end{definition}

\begin{proposition}[Homotopy lifting property]
Given a covering space $p:\tilde{X}\rightarrow X$, a homotopy $f_t:Y\times[0,1]\rightarrow X$ and a lifting $\tilde{f_0}:Y\rightarrow\tilde{X}$ of $f_0$, there is a unique homotopy $\tilde{f_t}:Y\times[0,1]\rightarrow\tilde{X}$ that lifts $f_t$.
\end{proposition}

\begin{proposition}
The induced map $p_\ast:\pi_1(\tilde{X},\tilde{x_0})\rightarrow\pi_1(X,x_0)$ is injective. The image subgroup $p_\ast(\pi_1(\tilde{X},\tilde{x_0}))$ consists of homotopy classes of loops in $X$ based at $x_0$ that lift to loops in $\tilde{X}$ based at $\tilde{x_0}$.
\end{proposition}

\begin{definition}
A space $X$ is semilocally simply connected if each point $x\in X$ has a neighborhood $U$ such that $\pi_1(U,x)\subset\pi_1(X,x)$ is trivial. 
\end{definition}

\begin{theorem}
If a space $X$ is path connected and locally path connected, then $X$ has a simply connected covering space if and only if $X$ is semilocally simple-connected.
\end{theorem}

\begin{theorem}
If $\tilde{X_1}\rightarrow X$ is a covering space and $\tilde{X}\rightarrow X$ is a simple-connected covering space, then $\tilde{X}$ is a covering space of $\tilde{X_1}$. Thus there is a partial ordering of covering spaces.
\end{theorem}

 The simply-connected covering space $\tilde{X}$ of $X$ is called the universal covering of $X$.
We will be only interested on Universal covers.\\

We introduce some basic facts about deck transformations.

\begin{definition}
An (self) isomorphism of covering spaces $\tilde{X}\rightarrow\tilde{X}$ is called a deck transformation. These forms a group $G(\tilde{X})$
\end{definition}

\begin{definition} 
A covering space $p:\tilde{X}\rightarrow X$ is normal if for each $x\in X$ and each pair of lifts $\tilde{x}, \tilde{x}'\in p^{-1}(x)$, there is a deck transformation taking $\tilde{x}$ to $\tilde{x}'$
\end{definition}

\begin{proposition}
Let $p:(\tilde{X},\tilde{x_0})\rightarrow(X,x_0)$ be a path-connected covering space of a path-connected, locally path-connected space $X$, and let 
$$H=p_\ast(\pi_1(\tilde{X},\tilde{x_0}))\leq\pi_1(X,x_0)$$
Then:
\begin{enumerate}
\item The group of deck transformations $G(\tilde{X})$ is isomorphic to $N(H)/ H$, where $N(H)$ is the normalizer subgroup.
\item The covering space is normal if and only if $H$ is a normal subgroup of $\pi_1(X,x_0)$
\end{enumerate}
\end{proposition}

\begin{corollary}
If $\tilde{X}$ is a normal covering, then $G(\tilde{X})\cong\pi_1(X,x_0)/H$. Thus if $\tilde{X}$ is the universal covering, then $G(\tilde{X})\cong\pi_1(X,x_0)$.
\end{corollary}

Thus if we have a group $\Gamma$ acting properly discontinuously on a simply-connected, locally path-connected space $M$, a base point $x_0\in M$ and we have the quotient map $p:M\rightarrow M/\Gamma$ then by corollary 4.19 we have that $\pi_1(M/\Gamma,p(x_0))\cong\Gamma$.   

\subsection{Cayley graph and Cayley complex}

In the last section we saw that we can realize any group $G$ as a fundamental group of some space. More precisely, given any group $G$ we are going to construct a simply-connected topological space $X$ such that $G$ acts free and proper discontinuously. Then by Corollary 4.19 we have that $G$ is the fundamental group of $X/G$.\\

To construct this space we proceed as follows: Let $G$ be a finitely generated and finitely presentable group, let $S=\{c_1,\ldots,c_k\}$ the generating set of $G$. Let's consider $A=S\cup S^{-1}$ where $S^{-1}=\{c_1^{-1},\ldots,c_k^{-1}\}$ is the set of formal inverses of the generating set $S$ (if there is an element $a\in S$ such that $a^2=1$, we take $a^{-1}\in S^{-1}$ as formal inverse). Let 
$$G=\langle A | R_1=\cdots=R_p=1\rangle$$ 
be a presentation of $G$, where $R_i$ are relations on elements of $A$ and consider the involution $\iota: A\rightarrow A$ given by $\iota(c_j)=\overline{c_j}$ for $j=1,\ldots,k$ where $\overline{c_j}=c_j^{-1}$ as element of $G$. We call this presentation admissible.

\begin{definition}
Suppose that we have an admissible presentation of the group $G$. Then, the Cayley graph of $G$ respect to the presentation is given by $C(G)=(V,E,\iota)$, where
\begin{itemize}
\item The set of vertices is given by $V=G$.
\item Two vertex $g,h\in G$ are connected by an edge if $g^{-1}h\in A$. Since $G$ is a group, then $g$ and $h$ are connected if and only if $h=ga$ for $a\in A$. Thus we say that $h$ and $ga$ are connected by a directed edge labelled by $a$.
\item The involution $\iota:A\rightarrow A$, is the involution which takes the edge which connects $h$ and $g$ labelled by $a$, with the edge which connects $g$ and $h$ labelled by $a^{-1}$.  
\end{itemize}
\end{definition} 

\begin{bf} Example:\end{bf} Let $F$ be a free group over the set $X$. Then $F$ has a presentation 
$$F=\langle\{x\}_{x\in X}|\emptyset\rangle$$
In order to have an admissible presentation we add a generator $x^{-1}$ for each $x\in X$ and we have
$$F=\langle \{x,x^{-1}\}_{x\in X}| xx^{-1}=x^{-1}x=e\rangle$$
Thus the vertices on $C(F)$ are labelled by the reduced words over the set of generators $\{x,x^{-1}\}_{x\in X}$, where a reduced word is a word in this letters without subwords of the form $xx^{-1}$ for $x\in X$. There is a geometric edge labelled by $xx^{-1}$ between $x_1x_2\cdots x_kx$ and $x_1x_2\cdots x_k$ where $x_1,\ldots,x_k\in X$. The corresponding Cayley graph is a Tree and hence its geometric realisation $|C(F)|_\iota$ is simply-connected (see Figure 25).

 \begin{figure}[h]
\includegraphics{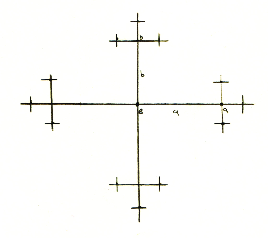}
\centering
\caption{Cayley graph of $F_2$}
\end{figure}

Let $R_j, 1\leq j\leq p$ be a relation in $G$, written as $R_j=a_{j_1}\cdots a_{j_k}$ where $a_{j_i}\in A$. Then any $g\in G$ satisfies 
$$gR_j=g(a_{j_1}\cdots a_{j_k})=g$$
thus there is a loop in $C(G)$ starting and ending at $g$ consisting of edges labelled by $a_{j_1},\ldots,a_{j_k}$ precisely in that order. In the geometric realization of $C(G)$ this loops are homeomorphic to circles and we can glue discs along this circles. The resulting space is called the Cayley 2-complex of $G$ whit respect the given presentation and denoted by $C_2(G)$.\\

\begin{bf}Example:\end{bf} Let's consider the group $\mathbb{Z}\times\mathbb{Z}$ whit the admisible presentation 
$$\langle a,b,a^{-1},b^{-1}| aba^{-1}b^{-1}=aa^{-1}=bb^{-1}=e\rangle$$
then $C_2(\mathbb{Z}\times\mathbb{Z})$ is as shown the following figure.

 \begin{figure}[h]
\includegraphics{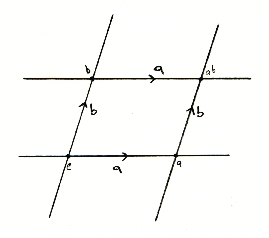}
\centering
\caption{Cayley 2-complex of $\mathbb{Z}\times\mathbb{Z}$}
\end{figure}

Observe that $G$ acts on itself by left action, thus $G$ acts on the set of vertices of $C(G)$. Extend this action into an action on $C(G)$ in the following way: if $k\in G$, then we send the edge which connects $g$ and $h$ to the edge which connects $kg$ and $kh$. This is well defined because if $g^{-1}h\in A$, then $(kg)^{-1}kh=g^{-1}k^{-1}kh=g^{-1}h\in A$. This action induce an action on the geometric realization $|C(G)|_\iota$ and extends to $C_2(G)$ by sending a disc attached to the loop corresponding to a relation $R_j$, to the disc attached to the loop corresponding to a relation $kR_j$. The action of $G$ on $C(G)$ (see \cite{Dru}) is free, transitive and if we have a neighborhood small enough, we will have at most the number of elements in a cycle satisfying $U\cap g(U)\neq\emptyset$. Since the cycles are finite, then we have that the action is proper discontinuous. Thus we have the following proposition.

\begin{proposition}
If $G$ is a group generated by $S$, then $C_2(G)$ is the universal covering of $X_G$, where $X_G$ is the space with $\pi_1(X_G)\cong G$ constructed by taking a wedge of circles, one for each generator in $S\cup S^{-1}$, and attaching a disc for each relation.
\end{proposition}
\begin{proof}
Let $p:C_2(G)\rightarrow C_2(G)/G$ by the quotient map given by identify the orbits of the action of $G$ on $C_2(G)$. Since $C_2(G)$ is arc-connected and locally arc-connected, since $S\cup S^{-1}$ is a generating set of $G$, then by corollary 4.19
$$G\cong\pi_1(C_2(G)/G)/p_\ast(\pi_1(C_2(G))).$$ 
Therefore, if we prove that $p_\ast(\pi_1(C_2(G)))$ is trivial, then we have $G\cong\pi_1(C_2(G)/G)$. To do this, we first identify $C_2(G)/G$ as $X_G$. Note that every vertex is identified in $C_2(G)/G$ because every group element is sent to any other group element because $S\cup S^{-1}$ is a generating set of $G$. Since every vertex in $C(G)$ has $|S|$ edges attached to it (one for every element of $S$), then we see that $C_2(G)/G$ is a wedge of $|S|$ many circles whit discs attached to corresponding relations on the generators. This is exactly $X_G$. Thus
$$C_2(G)/G\cong X_G.$$
Therefore, from above we also have 
$$G\cong\pi_1(X_G)/p_\ast(\pi_1(C_2(G))).$$
However, $X_G$ is constructed such that $\pi_1(X_G)\cong G$. It follows from 
$$\pi_1(X_G)\cong G\cong\pi_1(X_G)/p_\ast(\pi_1(C_2(G)))$$
that $p_\ast(\pi_1(C_2(G)))$ is trivial. Since $p:C_2(G)\rightarrow C_2(G)/G=X_G$ is a covering, then $p_\ast:\pi_1(C_2(G))\rightarrow\pi_1(X_G)$ is injective. Hence, $\pi_1(C_2(G))$ is trivial by above, and $C_2(G)$ is the universal covering for $X_G$. 
\end{proof}

\subsection{Classification of surfaces II: Unicity}

We apply the last proposition to prove the part of unicity of the classification theorem.

\begin{theorem}
The fundamental group of a surface $S_g$ is given by 
$$\pi_1(S_g)\cong\langle a_1,b_1,\ldots, a_g,b_g|\Pi_{i=1}^g a_ib_ia_i^{-1}b_i^{-1}=e\rangle.$$
These groups are non-isomorphic for different choices of $g$.
\end{theorem} 
\begin{proof}
For $g=0$ we have that $\pi_1(S_0)=\{e\}$ since $S_0$ is simple-connected. Let's assume that $g\geq1$ and we define
$$A_g:=\langle a_1,b_1,\ldots, a_g,b_g|\Pi_{i=1}^g a_ib_ia_i^{-1}b_i^{-1}=e\rangle$$

If we attach the inverse of the generators to this presentation and we construct the associated Cayley graph $C(A_g)$ and the Cayley 2-complex $C_2(A_g)$. Then by the proposition 4.21 we have that $C_2(A_g)/A_g$ is homeomorphic to a wedge of circles labelled by $a_1,b_1,a_1^{-1},b_1^{-1},\ldots,a_g,b_g,a_g^{-1},b_g^{-1}$, with a disc attached to them. This description is the same as the surface $S_g$, therefore by Proposition 4.21 we have $\pi_1(S_g)=A_g$.
\end{proof}


\section{Combinatorial description of the fundamental group using ribbon graphs}

In this section, we relate the filling ribbon graph of a surface and its fundamental group. We know that given a surface $S$ exists a filling ribbon graph. This ribbon graphs are not unique but we can deform this graphs to a ribbon graph of type $\Gamma_g$. Using different ribbon graphs we can compute the fundamental group. This gives us different presentations of the fundamental group.\\

Let $\Gamma$ be a filling ribbon graph. Denote by $E, V$ and $F$ the set of edges, vertices and faces respectively.

\begin{definition}[Combinatorial paths and loops]
A discrete path is a finite sequence $(e_1,\ldots, e_n)$ of edges such that $e_i^+=e_{i+1}^-$. The starting point of such path is $e_1^+$ and the ending point is $e_n^+$. A path is a discrete loop if its starting point and ending point are the same. We say that a loop has a base point at $v_0$, if $v_0$ is the starting and ending point. The inverse path of $e=(e_1,\ldots,e_n)$ is $\bar{e}=(\overline{e_n},\ldots,\overline{e_1})$. 
\end{definition}

Let $F(E)$ be the free group generated by the set of edges; note that every path defines an element of $F(E)$. Let $L_\Gamma^{v_0}$ be the image of the loops with base point $v_0$ on $F(E)$, then $L_\Gamma^{v_0}$ is a subgroup of $F(E)$. Let $R_\Gamma^{v_0}$ be the subgroup of $L_\Gamma^{v_0}$ normally generated by the faces $f=(e_1,\ldots,e_k)$.

\begin{definition} Let $\Gamma$ be a filling ribbon graph then: 
\begin{itemize}
\item The group $L_\Gamma^{v_0}$ is the group of loops with base point $v_0$.
\item The group $R_\Gamma^{v_0}$ is the group of homotopically trivial loops.
\item The ribbon fundamental group is $\hat{\pi_1}(\Gamma,v_0)=L_\Gamma^{v_0}/R_\Gamma^{v_0}$.
\item Two paths $e$ and $f$ with the same starting and final point are homotopics if the loop $e\bar{f}$ is an element of $R_\Gamma^{v_0}$.
\end{itemize}
\end{definition}

Now we relate the fundamental group of a surface $S$ with the ribbon fundamental group of its ribbon graph.

\begin{theorem}
Let $S$ be a closed surface and $i: |\Gamma|_I\rightarrow S$ be an embedding of the geometric realization of a filling ribbon graph $\Gamma$ on $S$. Then the natural mapping $i_\ast:\hat{\pi_1}(\Gamma, v_0)\rightarrow\pi_1(S,v_0)$, which send every combinatorial loop to its geometric realization, is an isomorphism of the ribbon fundamental group and the fundamental group of the surface.
\end{theorem} 
\begin{proof}
Since $\Gamma$ can be deformed to a filling ribbon graph $\Gamma_g$ which is a wedge of circles we have that $\hat{\Pi_1}(\Gamma, v_0)=\hat{\pi_1}(\Gamma_g,v_0)=A_g$ this by proposition 4.21. By theorem 4.22 we have that $\pi_1(S_g)=\pi_1(S)=A_g$. Therefore we have the following
$$\hat{\pi_1}(\gamma,v_0)=A_g=\pi_1(S)$$
and we are done.
\end{proof}
 
\section*{Acknowledgments}
The author would like to thank to A. Cano  for fruitful conversations, to J. R. Parker for his support and M. Montes de Oca Aquino and D. Enriquez for helping with the spelling correction.

\end{document}